\documentclass[a4paper,12pt]{article}
\usepackage[english]{babel}
\usepackage[top=3cm, bottom=3cm, left=2.3cm, right=2.3cm]{geometry}
\usepackage{url}
\usepackage{amsmath,amssymb,enumerate,latexsym,xcolor,theorem,dsfont, stmaryrd}
\usepackage{bbm,mathrsfs}
\usepackage{amsmath}
\usepackage{color}
\usepackage{hyperref}
\hypersetup{
    colorlinks=true,
    linkcolor=blue,
    citecolor=magenta,   
    urlcolor=blue,
    }
\newcommand{\assign}{:=}
\newcommand{\cdummy}{\cdot}
\newcommand{\mathLaplace}{\Delta}
\newcommand{\mathd}{\mathrm{d}}

\newcommand{\tmem}[1]{{\em #1\/}}
\newcommand{\tmop}[1]{\ensuremath{\operatorname{#1}}}

\newenvironment{enumerateroman}{\begin{enumerate}[i.] }{\end{enumerate}}
\newenvironment{proof}{\noindent\textbf{Proof.\ }}{\hspace*{\fill}$\Box$\medskip}
\newenvironment{proof*}[1]{\noindent\textbf{#1\ }}{\hspace*{\fill}$\Box$\medskip}
\newtheorem{theorem}{Theorem}
\newtheorem{definition}[theorem]{Definition}
\newtheorem{corollary}[theorem]{Corollary}
\newtheorem{lemma}[theorem]{Lemma}
{\theorembodyfont{\rmfamily}\newtheorem{remark}[theorem]{Remark}}

\begin{document}
\title{Energy solutions to SDEs with supercritical distributional drift: An extension and weak convergence rates}
\author{
  Lukas Gräfner\\
  University of Warwick\\
  \texttt{lukas.grafner@warwick.ac.uk}
}

\maketitle

\begin{abstract}\noindent
  In this work we consider the SDE
  \begin{equation}
   \mathd X_t = b (t, X_t) \mathd t + \sqrt{2} \mathd B_t, \label{mainSDE}
  \end{equation}
  in dimension $d \geqslant 2$, where $B$ is a Brownian motion and $b : \mathbb{R}_+ \rightarrow
  \mathcal{S}' (\mathbb{R}^d ; \mathbb{R}^d)$ is distributional, scaling
  super-critical and satisfies $\nabla \cdummy b \equiv 0$. We partially
  extend the super-critical weak well-posedness result for energy solutions
  from {\cite{GraefnerPerkowski24}} by allowing a mixture of the regularity
  regimes treated therein: Outside of neighbourhoods of a small (and compared to
  {\cite{GraefnerPerkowski24}} ``time-dependent'') local singularity set $K \subset \mathbb{R}_+
  \times \mathbb{R}^d$, $b$ is assumed to be in a certain supercritical $L^q_T H^{s, p}$-type class that allows a direct link between the PDE and the energy solution from a-priori estimates up
  to the stopping time of visiting $K$. To establish this correspondence, and
  thus uniqueness, globally in time we then show that $K$ is actually never
  visited which requires us to impose a relation between the dimension of $K$
  and the H{\"o}lder regularity of $X$.\\
  In the second part of this work we derive weak convergence rates for approximations of equation
  (\ref{mainSDE}) in the case of time-independent drift, in particular with local singularities as above.
  
  
\end{abstract}
\section{Introduction}

We assume that
\begin{equation}
  b (t)^i = \mathbf{b} (A (t))^i := \nabla \cdummy A (t)_i, \label{matrixdif}
\end{equation}
with $A :
\mathbb{R}_+ \times \mathbb{R}^d \rightarrow \mathbb{R}^{d \times d}$ an
antisymmetric matrix field (in particular $\nabla \cdummy b \equiv 0$) such that
\begin{equation}
  A \in L^q_{\tmop{loc}} L^p  \text{for some $p \left( \frac{1}{2} -
  \frac{1}{q} \right) > 1$, where $p \in [2, \infty)$ and $q \in [2, \infty]$}
  . \label{excond}
\end{equation}
Under this condition, which lies far within the
{\tmem{scaling-supercritical}} regime, existence of {\tmem{energy solutions}}
(see Definition \ref{defenergysolutionsEuclid}) is shown in
{\cite{GraefnerPerkowski24}} for any $\tmop{Leb}$- absolutely continuous
initial distribution. The authors then proceed to show two super-critical weak
well-posedness results for energy solutions under further assumptions. The
first result ({\cite[Theorem 1]{GraefnerPerkowski24}}) corresponds to
putting assumptions on the $L^q_T H^{s, p}$ regularity of $A$ while the second
one ({\cite[Theorem 4]{GraefnerPerkowski24}}) includes the case of
time-independent $A$ which is in $L^{\infty}$ away from a compact
$\tmop{Leb}_{\mathbb{R}^d}$-zero set on which $A$ is allowed to diverge in a
certain way. Let us also mention the recent and related work {\cite{HaoZhang23}} on SDEs
with supercritical distributional drift, where the authors consider a similar regime as the one of {\cite[Theorem 1]{GraefnerPerkowski24}}. They moreover reference {\cite{Takanobu85}} which inspired the present work: In the latter work the author
considers a case where $A$ is in $C^2$ away from a singularity at the origin.
They then use a stopping argument to solve the SDE classically away from the
singularity and show that the origin is almost surely not visited.\\
The purpose of this work is to adapt this idea in the case of ``globally distributional'', supercritical $b$ where classical notions of solution are not applicable. In this way, in Theorem \ref{maintheorem}, we partially extend the well-posedness results from {\cite{GraefnerPerkowski24}}
(in classes of H{\"o}lder-continuous solutions) by combining the two regimes
from {\cite[Theorem 1]{GraefnerPerkowski24}} and {\cite[Theorem
4]{GraefnerPerkowski24}}. In particular, compared to
{\cite{GraefnerPerkowski24}} we can now allow time-dependent singularities for
$A$ and our arguments for the proof of uniqueness are much more probabilistic
in nature. A key tool is Lemma \ref{stoppedob} which makes sense of the stopped object
\begin{equation*}
u(t\wedge\tau,X_{t\wedge\tau}) , 
\end{equation*}
for a certain stopping time $\tau$ and $u,X$ of very low regularity. We refer to Remark \ref{relationtoGP24} for a detailed discussion of Theorem \ref{maintheorem}.\\
In the second part of this work, we derive weak rates for the convergence of approximations to equation (\ref{mainSDE}): In Theorem \ref{ratewithoutlocalsingularities} we consider the setting of time-independent drift without local singularities and general initial densities, getting arbitrarily close to rate $1$ (see Remark \ref{rateremark}). In Corollary \ref{ratesingularitycorollary} we allow local singularities and consider bounded initial densities.
\paragraph{Plan of this work}
In Section \ref{Preliminaries} we specify the setting and state and discuss Theorem \ref{maintheorem}. Section \ref{maintheoremsection} contains the proof of Theorem \ref{maintheorem}. In Section \ref{ratessection} we derive weak convergence rates through Theorem \ref{ratewithoutlocalsingularities} and Corollary \ref{ratesingularitycorollary}.
\paragraph{Notation}
For numbers $a,b$ we write $a\lesssim b$ if $a\leqslant C b$ for some constant $C>0$ which is irrelevant for the discussion. To emphasize the dependence of $C$ on some parameter $T$, we write $a\lesssim_T b$. For equality up to such a constant we write $a\simeq b$. We denote by $\tmop{Leb}=\tmop{Leb}_{\mathbb{R}^d}$ the Lebesgue measure on $\mathbb{R}^d$.

\section{Preliminaries and main result}\label{Preliminaries}

For the sake of a slightly shorter write-up we mostly work with a finite
time-horizon $[0, T]$ but all results translate to $[0, \infty)$ in a
canonical way. For the same reason we restrict ourselves to the case that $b$ is divergence-free, although it would be possible to add non-divergence-free function-valued perturbations of critical regularity as in \cite{GraefnerPerkowski24}. We define
\[ \mathcal{A} = L^{\infty}_T L^{\infty} (\tmop{Skew}_d) + L^2_T H^1
   (\tmop{Skew}_d), \]
where $\tmop{Skew}_d$ is the space of antisymmetric elements of $\mathbb{R}^{d
\times d}$. It follows from Remark 2 in {\cite{GraefnerPerkowski24}} that
$L^r_T B^{- \gamma + 1}_{r, 1} (\tmop{Skew}_d) \subset \mathcal{A},$ for any $r
\in [2, \infty], \gamma \in [0, 1]$ such that $r \geqslant \frac{2}{1 -
\gamma}$. Here, $B^{- \gamma + 1}_{r, 1}$ is the inhomogeneous Besov space with spatial regularity $- \gamma + 1$, integrability $r$ and Besov parameter $1$. The threshold for scaling criticality is $r = \frac{d}{1 - \gamma}$
(see e.g. the introductions of {\cite{HaoZhang23}},
{\cite{GraefnerPerkowski24}}) so that $\mathcal{A}$ already corresponds to
super-critical regularity.\\
Now, we turn to our formalization of local singularities, on which $A$ is allowed to diverge away from the condition of being in $\mathcal{A}$. Let
\[ K \subset [0, T] \times \mathbb{R}^d, \]
be compact. We write $B_{\varepsilon} (K) = \{ (t, x) \in \mathbb{R} \times
\mathbb{R}^d : \inf_{(s, y) \in K} | (t, x) - (s, y) | < \varepsilon \}$. For $\varepsilon>0$ let
$g_\varepsilon \in C^\infty(\mathbb{R}\times\mathbb{R}^d)$  be such that
\begin{eqnarray*}
g_\varepsilon|_{B_{\varepsilon/2}(K)^c} \equiv 1,g_\varepsilon|_{B_{\varepsilon/4}(K)}\equiv 0 .
\end{eqnarray*}
This sequence exists by Lemma 33 in {\cite{GraefnerPerkowski24}}. We also write $B_{\varepsilon} (K)_t = \{ x \in \mathbb{R}^d : (t,
x) \in B_{\varepsilon} (K) \}$. Apart from (\ref{excond}) we put the following
standing assumption:
\begin{equation}
  \lim_{\varepsilon \rightarrow 0} \sup_{t \in [0, T]}
  \tmop{Leb}_{\mathbb{R}^d} (B_{\varepsilon} (K)_t) = 0, A g_{\varepsilon} \in
  \mathcal{A} \forall \varepsilon > 0 . \label{locallydiv}
\end{equation}
For a continuous stochastic process $X:[0,T]\rightarrow \mathbb{R}^d$ on some probability space we define the $\sigma (X)$-stopping
time
\[ \tau^{\infty} = \tau^{\infty} (X) \assign \inf \{ t \in [0, T] : (t, X_t)
   \in K \} , \]
   with the usual convention that $\inf \emptyset=\infty$.\\
The following definition of {\tmem{energy solutions}} is given in
{\cite{GraefnerPerkowski24}}.

\begin{definition}
  \label{defenergysolutionsEuclid}Let $b : \mathbb{R}_+ \rightarrow
  \mathcal{S}' (\mathbb{R}^d, \mathbb{R}^d)$ be such that for all $T > 0$ and
  $f \in C^{\infty}_c ([0, T] \times \mathbb{R}^d)$ we have $b\cdummy\nabla
  f \in L^2_T H^{- 1}$. Let $X$ be a stochastic process on a complete
  probability space $(\Omega, \mathcal{F}, \mathbb{P})$ with values in $C
  (\mathbb{R}_+ ; \mathbb{R}^d)$ such that for all $T > 0$:
  \begin{enumerateroman}
    \item \label{incompressibleEucild}$X$ is {\tmem{incompressible in
    probability}}, i.e. for all $\varepsilon > 0$ and $T > 0$ there exists $M
    = M (\varepsilon, T) > 0$ such that for all $B \in \mathcal{B}
    (\mathbb{R}^d)$
    \[ \mathbb{P} (X_t \in B) \leqslant \varepsilon + M \cdot \tmop{Leb} (B),
       \qquad t \in [0, T], \]
    and
    \[ \mathbb{P} \left( \sup_{t \leqslant T} \left| \int_0^t f (s, X_s)
       \mathd s \right| > \delta \right) \leqslant \varepsilon +
       \frac{M}{\delta} \| f \|_{L^1_T L^1 (\mathbb{R}^d)}, \]
    for all for all $f \in C^{\infty}_c ([0, T] \times \mathbb{R}^d)$.
    
    \item \label{admissibleEucild}X is satisfies an {\tmem{energy estimate in probability}}, i.e.
    for all $\varepsilon > 0$ and $T > 0$ there exists $M = M (\varepsilon, T)
    > 0$ such that for all $f \in C^{\infty}_c ([0, T] \times \mathbb{R}^d)$
    \begin{equation}
      \mathbb{P} \left( \sup_{t \leqslant T} \left| \int_0^t f (s, X_s) \mathd
      s \right| > \delta \right) \leqslant \varepsilon + \frac{M}{\delta} \| f
      \|_{L^2_T H^{- 1} (\mathbb{R}^d)} . \label{Itotrickbound}
    \end{equation}
    \item \label{martingalepropertyEuclid}For any $f \in C^{\infty}_c ([0, T]
    \times \mathbb{R}^d)$, the process
    \[ M^f_t = f (t, X_t) - f (0, X_0) - \int_0^t (\partial_s + \mathLaplace +
       b \cdummy \nabla) f (s, X_s) \mathd s, \qquad t \in [0, T], \]
    is a local martingale in the filtration generated by $X$,
    where the integral is defined as
    \[ \int_0^t (\partial_s + \mathLaplace + b \cdummy \nabla) f (s, X_s)
       \mathd s \assign I ((\partial_s + \mathLaplace + b \cdummy \nabla)
       f)_t, \]
    with $I$ the unique continuous extension from $C^{\infty}_c ([0, T] \times
    \mathbb{R}^d)$ to $L^2_T H^{- 1} (\mathbb{R}^d)$ of the map $g \mapsto
    \int_0^{\cdummy} g (s, X_s) \mathd s$ taking values in the continuous
    adapted stochastic processes equipped with the topology of uniform
    convergence on compacts (ucp-topology), and the extension $I$ exists by
    (\ref{Itotrickbound}).
    
    \item \label{quadraticvariationEuclid}The local martingale $M^f$ from
    \ref{martingalepropertyEuclid} has quadratic variation
    \[ \langle M^f \rangle_t = \int_0^t | \nabla f (s, X_s) |^2 \mathd s . \]
  \end{enumerateroman}
  Then $X$ is called an {\tmem{energy solution}} to the SDE $\mathd X_t = b
  (t, X_t) \mathd t + \sqrt{2} \mathd B_t$.
\end{definition}

Note that \ref{quadraticvariationEuclid} implies that the local martingale
from \ref{martingalepropertyEuclid} is a proper martingale. For $\alpha \in
(0, 1]$ we denote by $C^{\alpha}$ the space of H{\"o}lder-continuous
trajectories $X : [0, T] \rightarrow \mathbb{R}^d$ and put $C^{\alpha -} =
\bigcap_{\alpha > \varepsilon > 0} C^{\alpha - \varepsilon}$.
   
   We have the following well-posedness result.
\begin{theorem}
  \label{maintheorem}Assume that (\ref{excond}) and (\ref{locallydiv}) hold.
  \begin{enumerateroman}
    \item If for an energy solution $X$ to (\ref{mainSDE}) it holds
    \begin{equation}
      \tau^{\infty} (X) = \infty, \mathbb{P}- a.s., \label{stopinftyinfty}
    \end{equation}
    then the law of $X$ is uniquely determined by its initial distribution.
    Moreover, $X$ is a Markov process.
    
    \item \label{Hoelderassump}Set $\alpha = \frac{1}{2} - \frac{1}{q} -
    \frac{1}{p}$, with $p, q$ as in (\ref{excond}), and suppose that for some
    $\delta > 0$
    \[ \tmop{Leb}_{\mathbb{R}^{d + 1}} (B_{\varepsilon} (K)) \lesssim
       \varepsilon^{1 / \alpha + \delta} . \]
    Then (\ref{stopinftyinfty}) holds for any energy solution such that $X \in
    C^{\alpha -}$ almost surely. Such an energy solution exists for any $\mu =
    \tmop{law} (X_0) \ll \tmop{Leb}_{\mathbb{R}^d}$. Consequently, equation
    (\ref{mainSDE}) is well-posed in the class of energy solutions $X \in
    C^{\alpha -}$.
  \end{enumerateroman}
\end{theorem}

\begin{remark}
  \label{relationtoGP24}Assumption (\ref{locallydiv}) is a mixture of the
  regularities from {\cite[Theorem 1]{GraefnerPerkowski24}} and
  {\cite[Theorem 4]{GraefnerPerkowski24}}, however with time-dependent singularities: Outside of the local singularity
  set $K$, one has that $A \in \mathcal{A}$ which relates to the regime of
  the divergence-free part of the drift in {\cite[Theorem
  1]{GraefnerPerkowski24}}. In that regime the link between energy solutions
  to the SDE and solutions to the Kolmogorov backward equation (KBE) is established by checking
  that the regularities of SDE, PDE and the drift operator are compatible and
  we use the same argument e.g. in the poof of Lemma \ref{PDEregularity} and
  in (\ref{addfunclimit}) below.\\
  \\
  Near $K$ however, the field $A$ is allowed to diverge away from this regime.
  This corresponds to the {\tmem{time-homogeneous}} case of {\cite[Lemma
  5]{GraefnerPerkowski24}} which gives a sufficient condition for
  {\cite[Theorem 4]{GraefnerPerkowski24}} and assumes the existence of a
  compact set $L \subset \mathbb{R}^d$ with
  \begin{equation}
    \sup_{\varepsilon > 0} \varepsilon^{- 2} \tmop{Leb} (B_{\varepsilon} (L)),
    \qquad \sup_{\varepsilon > 0} \varepsilon^{- 2} \int_{B_{\varepsilon} (L)}
    | A |^2 < \infty,
  \end{equation}
  and for all $\varepsilon > 0$,
  \begin{equation}
    A 1_{B_{\varepsilon}^c (L)} \in L^{\infty} .
  \end{equation}
  The analysis in {\cite{GraefnerPerkowski24}} then is more
  functional-analytic and focuses on constructing a unique semigroup $(P_t)$
  on $L^2 (\tmop{Leb})$ which has sufficient properties to conclude the
  correspondence to the process $X$. This approach is now replaced with a
  probabilistic argument involving the stopping times of visiting
  $\varepsilon$-neighbourhoods of $K$. Note that although we can treat the
  time-inhomogeneous case now, the condition $\lim_{\varepsilon \rightarrow 0}
  \sup_{t \in [0, T]} \tmop{Leb}_{\mathbb{R}^d} (B_{\varepsilon} (K)_t) = 0$
  in (\ref{locallydiv}) imposes some kind of time-regularity on the set $K$.\\
  \\
  Assume that $A$ is time-independent and $A \in L^p$. Then existence of
  energy solutions $X$ with $X \in C^{\alpha}$ for all $\alpha < \frac{p / 2 - 1}{p}$
  follows from the proof of Theorem 19 in {\cite{GraefnerPerkowski24}} for
  any initial distribution $\mu \ll \tmop{Leb}$. If we assume that $K = [0, T]
  \times \{ 0 \}$, then
  \[ \tmop{Leb} (B_{\varepsilon} (K)) \lesssim \varepsilon^d, \]
  so that condition \ref{Hoelderassump} from Theorem \ref{maintheorem} becomes
  \[ \left( \frac{p / 2 - 1}{p} \right)^{- 1} < d, \]
  and thus
  \[ p > \frac{2 d}{d - 2} . \]
  This is the same dimension-dependent lower bound as in {\cite[Example
  6]{GraefnerPerkowski24}}, which allows to treat proper elements $b \in
  W^{- 1, p, \tmop{loc}}$ for any $p > 2$ in high enough dimensions. Thus we
  do not expect to enter a fundamentally new regime of regularity with Theorem
  \ref{maintheorem}, at least not without imposing further structural
  conditions on $b$ as in {\cite{Takanobu85}}.
\end{remark}

\begin{remark}
  In {\cite{Takanobu85}}, the author considers a drift of the form $b (X) =
  \sum_{i \neq j} \gamma_j \nabla^{\perp} H  (X_i - X_j)$ with $X_i \in
  \mathbb{R}^d, 1 \leqslant i \leqslant n$, where $H (x) = g (| x |), x \in
  \mathbb{R}^d$, $g \in C^2 (0, \infty)$ and $g$ diverges at the origin in a
  specific way, which allows $g$ to even be non-integrable at $0$ via a
  fine-grained analysis of the corresponding stopping time $\tau^{\infty}$.
  Non-locally integrable (in space) $A$ is of course forbidden in our case
  since we require $\mathbf{b} (A)$ to at least be well-defined as a
  distribution and we need at least local $L^2$-integrability, even for
  existence of energy solutions.
\end{remark}

\section{Proof of Theorem \ref{maintheorem}}\label{maintheoremsection}

In order to not confuse notation with the $\varepsilon$ from Definition
\ref{defenergysolutionsEuclid}, we will use $\delta$ to describe closeness to
the set $K$ from now on. For an energy solution $X$ we define \[\tau^{\delta} =
\inf \{ t \in [0, T] : (t, X_t) \in \overline{B_{\delta} (K)} \}.\]Note that
$\tau^{\infty} = \lim_{\delta \rightarrow 0} \tau^{\delta}$. In the following
we often make use of the identity
\begin{equation}
  \mathbf{b} (A) \cdummy \nabla u = \nabla \cdummy (\mathbf{b} (A) u) = \nabla
  \cdummy (A \nabla u), \label{centralidentity}
\end{equation}
for antisymmetric $A$ and $u$ with sufficient regularities. Another property that we will frequently use without mentioning is that $g_\delta$ is
bounded and $\partial_t g_\delta, \nabla g_\delta$ are compactly supported which follows
both from the fact that $g_\delta \equiv 1$ outside of a compact set.\\
The next lemma establishes that solutions $v$ to the KBE for equation \ref{mainSDE} are in $C_T L^2$ so that we can in particular control compositions $v(t,X_t)$ with an energy solution $X$ uniformly in $t$.

\begin{lemma}
  \label{PDEregularity}For any $v_0 \in \mathcal{S}$ there exists a weak
  solution in the distributional sense in space and time $v \in L^2_T H^1 \cap
  L^{\infty}_T L^{\infty} \cap C_T L^2$ to
  \[ \partial_t v = \mathLaplace v + \nabla \cdummy (A \nabla v), v (0) = v_0
     . \]
\end{lemma}

\begin{proof}
  Note that $\nabla \cdummy (A \nabla v)$ is well-defined since $A, \nabla v
  \in L^2_T L^{2, \tmop{loc}}$. Based on the approximated solutions
  \[ \partial_t v^n = \mathLaplace v^n + \nabla \cdummy (A^n \nabla v^n), v^n
     (0) = v_0, \]
  where $A^n = \rho^n \ast A$ and $\rho^n$ is a space-time mollifier,
  existence of $v$ for the case $A \in L^{\infty}_T L^{\infty} + L^2_T H^1$ is
  shown in {\cite{GraefnerPerkowski24}}. However it is rather straightforward
  to check that the argument therein works for our case of $A \in L^q_T L^{p}$ as in (\ref{excond}), at least to obtain a solution $v \in L^2_T H^1 \cap
  L^{\infty}_T L^{\infty} \cap C_T \mathcal{S}'$, see also the proof in
  {\cite{HaoZhang23}}. Indeed, finding a subsequential limit $v \in L^2_T H^1
  \cap L^{\infty}_T L^{\infty}$ that satisfies the solution property in the
  space-time distributional sense follows in the same way as in {\cite{GraefnerPerkowski24}}.    Moreover, given $f \in \mathcal{S}$, we have for $N
  = N (d) > 0$ big enough such that $(1 + | \mathbf{x} |)^{- N} \in L^1$ that
  \begin{eqnarray*}
    | \langle v^n (t), f \rangle - \langle v^n (s), f \rangle | & \leqslant &
    \int_s^t | \langle \nabla f, (\tmop{Id} + A^n (r)) \nabla v^n (r) \rangle
    | \mathd r\\
    & \leqslant & \| (1 + | \mathbf{x} |)^N \nabla f \|_{L^{\infty}} \\
    &&\times \int_s^t
    \| (1 + | \mathbf{x} |)^{- N} (\tmop{Id} + A^n (r)) \nabla v^n (r)
    \|_{L^1} \mathd r \\
    & \lesssim_{N (d), T} & \| (1 + | \mathbf{x} |)^N \nabla f
    \|_{L^{\infty}} \sup_n \| v^n \|_{L^2_T H^1} \int_s^t (1 + \| A^n (r)
    \|_{L^p})^q \mathd r\\
    & \lesssim & \| (1 + | \mathbf{x} |)^N \nabla f \|_{L^{\infty}} \sup_n \|
    v^n \|_{L^2_T H^1} (| t - s | + \| A |_{[s, t]} \|_{L^q
    L^p}^q),
  \end{eqnarray*}
  which shows that $(\langle v^n, f \rangle)_{n \in \mathbb{N}}$ is
  equicontinuous and thus by the Arzela-Ascoli theorem and the fact that we
  may assume weak convergence in $L^2_T H^1$ to $v$, we obtain a unique limit
  and the above estimate and the fact that $v_0 \in \mathcal{S}$ allows us to
  conclude that $v \in C_T \mathcal{S}'$ with $v (0) = v_0$.\\In order to obtain $v \in C_T L^2$, we first note that
  actually
  \begin{equation}
    \sup_{t \geqslant 0} \| v (t) \|_{L^{\infty}} < \infty,
    \label{alltbounded}
  \end{equation}
  since for a full-measure and hence dense set $D \subset [0, T]$ we have
  $\sup_{t \in D} \| v (t) \|_{L^{\infty}}<\infty$ and thus for any $t \in [0, T]$
  there exists $t_m \rightarrow t$ with $t_m \in D$, meaning that for $f \in
  \mathcal{S}$,
  \[ | \langle v (t), f \rangle | = \lim_m | \langle v (t_m), f \rangle |
     \leqslant \sup_{s \in D} \| v (s) \|_{L^{\infty}} \| f \|_{L^1}, \]
  and this implies that $v (t) \in (L^1)^{\ast} = L^{\infty}$ with $\| v (t)
  \|_{L^{\infty}} \leqslant \sup_{s \in D} \| v (s) \|_{L^{\infty}}$. Now, for
  every $\delta > 0$ we have $g_{\delta} v \in L^2_T H^1 \cap H^1_T H^{- 1}$.
  Indeed, we only have to note that in the weak sense
  \[ \partial_t (g_{\delta} v) = v \partial_t g_{\delta} + g_{\delta} \partial_t
     v = v \partial_t g_{\delta} + g_{\delta} (\mathLaplace v + \nabla \cdummy
     (g_{\delta / 4} A \nabla v)) . \]
  Since $g_{\delta / 4} A \in \mathcal{A}$ we have $g_{\delta / 4} A = A_1 +
  A_2$ for some $A_1 \in L^{\infty}_T L^{\infty} (\tmop{Skew}_d), A_2 \in
  L^2_T H^1 (\tmop{Skew}_d)$. Now we can estimate
  \[ \| \nabla \cdummy (A_1 \nabla v) \|_{L^2_T H^{- 1}} \lesssim \| A_1
     \|_{L^{\infty}_T L^{\infty}} \| v \|_{L^2_T H^1}, \]
  and using $\nabla \cdummy (A_2 \nabla v) = \nabla \cdummy (\mathbf{b} (A_2)
  v)$, which can be checked by approximation, we also have
  \[ \| \nabla \cdummy (\textbf{} A_2 \nabla v) \|_{L^2_T H^{- 1}} \lesssim \|
     \mathbf{b} (A_2) v \|_{L^2_T L^2} \lesssim {\| A_2 \|_{L^2_T H^1}}  \| v
     \|_{L^{\infty}_T L^{\infty}}, \]
  so that indeed $\partial_t g_{\delta} v \in L^2_T H^1$. Theorem 3 in
  {\textsection} 5.9.2 of Part II in {\cite{Evans98}} thus yields $g_{\delta}
  v \in C_T L^2$. Finally to conclude that $v \in C_T L^2$ we fix some $t \in
  [0, T]$ and write
  \[ \| v (t) - v (s) \|_{L^2} \leqslant \| g_{\delta} v (t) - g_{\delta} v
     (s) \|_{L^2} + \| 1_{B_{\delta} (K)} v (t) \|_{L^2} + \| 1_{B_{\delta}
     (K)} v (s) \|_{L^2} . \]
  Using (\ref{alltbounded}) the second and third term can be bounded by
  \[ \| 1_{B_{\delta} (K)} v (t) \|_{L^2} \lesssim \tmop{Leb}_{\mathbb{R}^d}
     (B_{\delta} (K)_t)^{1/2}, \]
  and this bound can be made arbitrarily small, uniformly in $t$, by choosing
  $\delta$ small by (\ref{locallydiv}).
\end{proof}

In the following we will refer several times to the properties \ref{incompressibleEucild}-\ref{quadraticvariationEuclid} from the definition of energy solutions and denote them in that way. Our goal is to see how we can use the fact that $A$ and the much better object $g_\delta A$ look the same for the stopped process $X_{t\wedge \tau^\delta}$ in order to get better control on the corresponding stopped additive functional.

Let $T' \in [0, T]$, consider a weak solution $v \in L^2_T H^1 \cap
L^{\infty}_T L^{\infty} \cap C_T L^2$ on $[0, T']$ to
\[ \partial_t v = \mathLaplace v + \nabla \cdummy (A (T' - \cdummy) \cdummy
   \nabla v), v (0) \in \mathcal{S} \]
and set $u = v (T' - \cdummy)$ as well as $u^{\gamma} = \rho^{\gamma} \ast u$,
where $\rho^{\gamma}=\gamma^{-d-1}\rho(\gamma^{-1}\cdot)$ and $\rho$ is a smooth mollifier in space and time.

Since $\nabla u^{\gamma} \in C_T L^r$ for any $r \in [2, \infty]$ we have $b
\cdummy \nabla u^{\gamma} = \nabla \cdummy (A \nabla u^{\gamma}) \in L^2_T
H^{- 1}$ and thus we can write
\begin{eqnarray*}
&&\int_0^{t \wedge \tau^{\delta}} (\partial_t + \mathLaplace + b \cdummy
   \nabla) u^{\gamma} (s, X_s) \mathd s\\
 &=&\int_0^{t \wedge \tau^{\delta}}
   (\mathLaplace + \nabla \cdummy (A \nabla \cdummy)) u^{\gamma} (s, X_s) -
   \rho^{\gamma} \ast [(\mathLaplace + \nabla \cdummy (A \nabla \cdummy)) u]
   (s, X_s) \mathd s,  
\end{eqnarray*}
where the last term comes from the fact that $\rho^{\gamma} \ast$ and
$\partial_t$ commute.

Since $\rho^{\gamma} \ast [(\mathLaplace + \nabla \cdummy (A \nabla \cdummy))
u] \in C^{\infty}_T C^{\infty}_b$ it is straightforward to check with an
approximation by $C^{\infty}_c$ functions that the (stopped) additive
functional $I (\rho^{\gamma} \ast [(\mathLaplace + \nabla \cdummy (A \nabla
\cdummy)) u])_{t \wedge \tau^{\delta}}$ appearing in the last line is equal to
the actual well-defined integral over the corresponding process. Thus, using the definition of $\tau^\delta$, we have
\[ 1_{\tau^\delta >0}I (\rho^{\gamma} \ast [(\mathLaplace + \nabla \cdummy (A \nabla \cdummy))
   u])_{t \wedge \tau^{\delta}} =1_{\tau^\delta>0} \int_0^{t \wedge \tau^{\delta}}
   \rho^{\gamma} \ast [(\mathLaplace + \nabla \cdummy (g_{\delta} A \nabla
   \cdummy)) u] (s, X_s) \mathd s, \]
for $\gamma (\delta)$ small enough such such that $\gamma \sup \{|(t,x)|:(t,x)\in \text{supp}\rho\} < \frac{\delta}{2}$. The integral on the r.h.s. is now also
a-priori well-defined as (and equal to) $I (\rho^{\gamma} \ast [(\mathLaplace
+ \nabla \cdummy (g_{\delta} A \nabla \cdummy)) u])_{t \wedge \tau^{\delta}}$,
even without the mollification, since $g_{\delta} A = A_1 + A_2$ for some $A_1
\in L^{\infty}_T L^{\infty} (\tmop{Skew}_d), A_2 \in L^2_T H^1
(\tmop{Skew}_d)$ and then $\nabla \cdummy (g_{\delta} A \nabla u) \in L^2_T
H^{- 1}$ due to (\ref{centralidentity}) as in the proof of Lemma
\ref{PDEregularity}. Arguing in a similar way we can approximate $C^{\infty}_c
\ni A^n \rightarrow A$ in $L^q_T L^p$ and $C^{\infty}_c
\ni u^{\gamma, n} \rightarrow
u^{\gamma}$ in $C_T H^{1 + \varepsilon, p'}$ where $\frac{1}{p'} + \frac{1}{p}
= \frac{1}{2}$, meaning that
\[
\nabla\cdot (A^n\nabla u^{\gamma, n})\rightarrow \nabla\cdot (A\nabla u^{\gamma}) ,
\]
to get that
\begin{eqnarray*}
  I ((\mathLaplace + \nabla \cdummy (A \nabla \cdummy)) u^{\gamma} (s,
  X_s))_{t \wedge \tau^{\delta}}  & =_{\ref{martingalepropertyEuclid}} &
  \lim_{\tmop{prob}} \int_0^{t \wedge \tau^{\delta}} (\mathLaplace + \nabla
  \cdummy (A^n \nabla \cdummy)) u^{\gamma, n} (s, X_s) \mathd s\\
  & = & \lim_{\tmop{prob}} \int_0^{t \wedge \tau^{\delta}} (\mathLaplace +
  \nabla \cdummy (g_{\delta} A^n \nabla \cdummy)) u^{\gamma, n} (s, X_s) \mathd s\\
  & =_{\ref{martingalepropertyEuclid}} & I ((\mathLaplace + \nabla \cdummy
  (g_{\delta} A \nabla \cdummy)) u^{\gamma} (s, X_s))_{t \wedge \tau^{\delta}}
  ,
\end{eqnarray*}
where $\lim_{\tmop{prob}}$ denotes the limit in probability. \\In total, we obtain, now writing $I$ symbolically as an integral again, as in the rest
of this paper,
\begin{eqnarray}
  \int_0^{t \wedge \tau^{\delta}} (\partial_t + \mathLaplace + b \cdummy
  \nabla) u^{\gamma} (s, X_s) \mathd s & = & \int_0^{t \wedge \tau^{\delta}}
   \nabla \cdummy (g_{\delta} A \nabla u^{\gamma}) (s,
  X_s) \nonumber\\
  &  & - \rho^{\gamma} \ast [ \nabla \cdummy (g_{\delta} A
  \nabla u) ] (s, X_s) \mathd s,  \label{convint}
\end{eqnarray}
for $\gamma (\delta)$ small enough. It follows from a straightforward
approximation argument using {\cite[Proposition~IX.1.17]{Jacod2003}} that the
martingale property \ref{martingalepropertyEuclid} from Definition
\ref{defenergysolutionsEuclid} and the statement
\ref{quadraticvariationEuclid} on the quadratic variation hold for
$u^{\gamma}$ as test functions even though they are not compactly supported
and we still write $M^{u^{\gamma}}$ etc. for the corresponding (proper!)
martingales.\\
The goal of the next Lemma is to make sense of the stopped object
\[ u(t\wedge\tau^\delta,X_{t\wedge\tau^\delta}) ,
\] 
which will appear in the proof of Theorem \ref{maintheorem}. This is a non-trivial problem due to the low regularities of $u$ and $X$, e.g. it is not clear if $\text{law} (X_{t\wedge\tau^\delta})\ll\text{Leb}$, so that even $u(t,X_{t\wedge\tau^\delta})$ might not be well-defined a-priori.

\begin{lemma}
  \label{stoppedob}Let $\delta > 0$. Then,
  \[ (u^{\gamma} (\cdot \wedge \tau^{\delta}, X_{\cdot \wedge \tau^{\delta}}))_{\gamma
     \in \mathbb{R}}, \]
  is Cauchy in $\tmop{ucp}$-topology and consequently in $L^p (C ([0, T']))$
  for any $1 \leqslant p < \infty$. Moreover,
  \[ u (t \wedge \tau^{\delta}, X_{t \wedge \tau^{\delta}}) \assign
     \lim_{\gamma} u^{\gamma} (t \wedge \tau^{\delta}, X_{t \wedge
     \tau^{\delta}}), \]
  satisfies
  \[ u (t \wedge \tau^{\delta}, X_{t \wedge \tau^{\delta}}) \leqslant \| u
     \|_{L^{\infty}_T L^{\infty}}, t \in [0, T'], \mathbb{P}- a.s. \]
  and
  \begin{equation}
    1_{\tau^{\delta} \geqslant t} u (t \wedge \tau^{\delta}, X_{t \wedge
    \tau^{\delta}}) = 1_{\tau^{\delta} \geqslant t} u (t, X_t),
    \mathbb{P}- a.s., t \in [0, T'] . \label{stoppedident}
  \end{equation}
\end{lemma}

\begin{proof}
  In the following we write $u^{\gamma_1, \gamma_2} = u^{\gamma_1} -
  u^{\gamma_2}$. We have that
  \begin{eqnarray*}
  & & \mathbb{P} (\sup_{t \in [0, T']} | u^{\gamma_1, \gamma_2} (t \wedge
    \tau^{\delta}, X^{\tau^{\delta}}) | > \kappa)\\
     & \leqslant & \mathbb{P}
    \left( | u^{\gamma_1, \gamma_2} (0, X_0) | > \frac{\kappa}{3}, | X_0 |
    \leqslant L \right)\\
    &  & +\mathbb{P} (| X_0 | > L)\\
    &  & +\mathbb{P} \left( \sup_{t \in [0, T']} \left| \int_0^{t \wedge
    \tau^{\delta}} (\partial_s + \mathLaplace + b \cdummy \nabla) u^{\gamma_1,
    \gamma_2} (s, X_s) \mathd s \right| > \frac{\kappa}{3} \right)\\
    &  & +\mathbb{P} \left( \sup_{t \in [0, T']} \left| M^{u^{\gamma_1,
    \gamma_2}}_{t \wedge \tau^{\delta}} \right| > \frac{\kappa}{3}, \right) .
  \end{eqnarray*}
  We bound each term an redefine the value of $M (\varepsilon)$ below along
  the way if necessary.
  
  By incompressibility, we can bound
  \begin{eqnarray*}
    \mathbb{P} \left( | u^{\gamma_1, \gamma_2} (0, X_0) | > \frac{\kappa}{3},
    | X_0 | \leqslant L \right) & \leqslant & \varepsilon / 5 + M
    (\varepsilon) \tmop{Leb} \left( 1_{\overline{B_L (0)}} | u^{\gamma_1,
    \gamma_2} (0) | > \frac{\kappa}{3} \right)\\
    & \leqslant & \varepsilon / 5 + \frac{3 M (\varepsilon)}{\kappa} \left\|
    1_{\overline{B_L (0)}} u^{\gamma_1, \gamma_2} (0) \right\|_{L^1}\\
    & \lesssim_L & \varepsilon / 5 + \frac{3 M (\varepsilon)}{\kappa} \|
    u^{\gamma_1, \gamma_2} \|_{C_T L^2} .
  \end{eqnarray*}
  Choosing $L$ large enough we have
  \[ \mathbb{P} (| X_0 | > L) \leqslant \varepsilon / 5 . \]
  Furthermore,
  \begin{eqnarray*}
    \mathbb{P} \left( \sup_{t \in [0, T']} \left| \int_0^{t \wedge
    \tau^{\delta}} (\partial_s + \mathLaplace + b \cdummy \nabla) u^{\gamma_1,
    \gamma_2} (s, X_s) \mathd s \right| > \frac{\kappa}{3} \right)&&\\
    \leqslant_{\left( \ref{convint} \right), \ref{admissibleEucild}}
    \frac{\varepsilon}{5} + \frac{3 M (\varepsilon)}{\kappa} \| \nabla \cdummy (g_{\delta} A \nabla u^{\gamma_1, \gamma_2})
    \|_{L^2_T H^{- 1}}&& \\
    + \frac{3 M (\varepsilon)}{\kappa} \| (\rho^{\gamma_1}
    - \rho^{\gamma_2}) \ast ( \nabla \cdummy (g_{\delta} A
    \nabla u))  \|_{L^2_T H^{- 1}} . &&
  \end{eqnarray*}
  Writing again $g_{\delta} A = A_1 + A_2$ as above, and arguing as in the
  proof of Lemma \ref{PDEregularity}, we have that $
  \rho^{\gamma} \ast \nabla \cdummy (g_{\delta} A \nabla u)$ is Cauchy
  in $L^2_T H^{- 1}$. Also $ \nabla \cdummy (g_{\delta} A
  \nabla  u^{\gamma})$ is Cauchy in this space for the same reason. We
  only have to be careful with the $A_2$-term since $u^{\gamma} \rightarrow u$
  only a.e. uniformly boundedly in $L^{\infty}_T L^{\infty}$ but this is
  sufficient as we can apply the dominated convergence theorem to
  \[ \| \nabla \cdummy (\textbf{} A_2 \nabla (u - u^{\gamma})) \|_{L^2_T H^{-
     1}}^2 \lesssim \int | \mathbf{b} (A_2) (u - u^{\gamma}) |^2 \mathd x \mathd t, \]
  using $\mathbf{b} (A_2) \in L^2_T L^2$.
  
  Finally, defining the stopping time $\sigma = \inf \left\{ t \leqslant T :
  \langle M^{u^{\gamma_1, \gamma_2}} \rangle_t \geqslant 2 \kappa^3 \right\}$
  we have by incompressibility,
  \begin{eqnarray*}
    \mathbb{P} \left( \sup_{t \in [0, T']} \left| M_{t \wedge
    \tau^{\delta}}^{u^{\gamma_1, \gamma_2}} \right| > \frac{\kappa}{3} \right)
    & \leqslant & \mathbb{P} \left( \langle M^{u^{\gamma_1, \gamma_2}}
    \rangle_T > \kappa^3 \right) \\
    & & +\mathbb{P} \left( \sup_{t \in [0, T']}
    \left| M_{t \wedge \tau^{\delta}}^{u^{\gamma_1, \gamma_2}} \right| >
    \frac{\kappa}{3}, \langle M^{u^{\gamma_1, \gamma_2}} \rangle_T \leqslant
    \kappa^3 \right)\\
    & \leqslant & \frac{\varepsilon}{5} + \frac{M (\varepsilon)}{\kappa^3} \|
    (\nabla u^{\gamma_1, \gamma_2})^2 \|_{L^1_T L^1}\\
    &&+\mathbb{P} \left(
    \sup_{t \in [0, T']} \left| M_{t \wedge \tau^{\delta}}^{u^{\gamma_1,
    \gamma_2}} \right| > \frac{\kappa}{3}, \sigma = \infty \right)\\
    & \lesssim & \frac{\varepsilon}{5} + \frac{M (\varepsilon)}{\kappa^3} \|
    u^{\gamma_1, \gamma_2} \|_{L^2_T H^1}^2 +\mathbb{P} \left( \sup_{t \in [0,
    T']} \left| M_{t \wedge \tau^{\delta} \wedge \sigma}^{u^{\gamma_1,
    \gamma_2}} \right| > \frac{\kappa}{3} \right)\\
    & \lesssim & \frac{\varepsilon}{5} + \frac{M (\varepsilon)}{\kappa^3} \|
    u^{\gamma_1, \gamma_2} \|_{L^2_T H^1}^2 + \frac{3}{\kappa} \mathbb{E}
    \left[ \left( \langle M^{u^{\gamma_1, \gamma_2}} \rangle_{T' \wedge
    \tau^{\delta} \wedge \sigma} \right)^{1 / 2} \right]\\
    & \leqslant & \frac{\varepsilon}{5} + \frac{M (\varepsilon)}{\kappa^3} \|
    u^{\gamma_1, \gamma_2} \|_{L^2_T H^1}^2 + \frac{3}{\kappa} (2 \kappa^3)^{1
    / 2},
  \end{eqnarray*}
  where we applied Markov- and then BDG inequality in the fourth step.
  
  This bound means that given $\kappa > 0$ and $\varepsilon > 0$ we can choose
  $\tilde{\kappa} = \tilde{\kappa} (\varepsilon, \kappa) < \kappa$ small
  enough such that $\frac{3}{\tilde{\kappa}} (2 \tilde{\kappa}^3)^{1 / 2} <
  \frac{\varepsilon}{5}$ so that, redefining $M (\varepsilon)$ along the way,
  \begin{eqnarray*}
    \mathbb{P} \left( \sup_{t \in [0, T']} \left| M_{t \wedge
    \tau^{\delta}}^{u^{\gamma_1, \gamma_2}} \right| > \frac{\kappa}{3} \right)
    & \leqslant & \mathbb{P} \left( \sup_{t \in [0, T']} \left| M_{t \wedge
    \tau^{\delta}}^{u^{\gamma_1, \gamma_2}} \right| > \frac{\tilde{\kappa}}{3}
    \right)\\
    & \lesssim & \frac{2 \varepsilon}{5} + \frac{M
    (\varepsilon)}{\tilde{\kappa}^3 (\varepsilon, \kappa)} \| u^{\gamma_1,
    \gamma_2} \|_{L^2_T H^1}^2\\
    & = & \frac{2 \varepsilon}{5} + M (\varepsilon, \kappa) \| u^{\gamma_1,
    \gamma_2} \|_{L^2_T H^1}^2 .
  \end{eqnarray*}
  Combining all estimates we get the Cauchy property for $(u^{\gamma} (\cdummy
  \wedge \tau^{\delta}, X_{\cdummy \wedge \tau^{\delta}}))_{\gamma}$ in the
  $\tmop{ucp}$-topology. At the same time we have
  \[ | u^{\gamma} (t \wedge \tau^{\delta}, X_{t \wedge \tau^{\delta}}) |
     \leqslant \| u \|_{L^{\infty}_T L^{\infty}}, t \in [0, T'], \mathbb{P}-
     a.s, \]
  which implies convergence in $L^p (C [0, T'])$ for any $1 \leqslant p <
  \infty$ and this bound is inherited in the limit. Also the identity
  \[ 1_{\tau^{\delta} \geqslant t} u^{\gamma} (t \wedge \tau^{\delta}, X_{t
     \wedge \tau^{\delta}}) = 1_{\tau^{\delta} \geqslant t} u^{\gamma} (t,
     X_t), \]
  is inherited by the limit in this way.
\end{proof}
\\The final ingredient for the proof of Theorem \ref{maintheorem} is the following lemma which tells us that if the “dimension” of $K$ is small enough, it is actually never visited by $X$, thus allowing us to remove $\tau^\delta$ by taking $\delta\rightarrow 0$. Since the proof essentially only uses the Hölder regularity of $X$ and incompressibility, we do not expect this to be a new result.
\begin{lemma}
  \label{relationdimHoelder}Assume that almost-surely $X \in C^{\alpha}$ for
  some $\alpha \in (0, 1]$ and that for some $\delta > 0$
  \[ \tmop{Leb}_{\mathbb{R}^{d + 1}} (B_{\varepsilon} (K)) \lesssim
     \varepsilon^{1 / \alpha + \delta} . \]
  Then, $\mathbb{P}$-a.s.,
  \[ \tau^{\infty} = \infty . \]
\end{lemma}

\begin{proof}
  Let $\omega \in \Omega$ such that $\tau^{\infty} (\omega) < \infty$, meaning
  that for some $t \in [0, T]$ it holds $(t, X_t (\omega)) \in K$. Then, still
  $(s, X_s (\omega)) \in B_{\varepsilon} (K)$ if
  \[ | t - s | < \left( \| X (\omega) \|_{C^{\alpha}}^{- 1}
     \frac{\varepsilon}{\sqrt{2}} \right)^{1 / \alpha}, | t - s | <
     \frac{\varepsilon}{\sqrt{2}}, \]
  meaning that for $\varepsilon > 0$ small enough
  \[ \tmop{Leb} \{ s \in [0, T] : (s, X_s (\omega)) \in B_{\varepsilon} (K) \}
     \geqslant \min \left\{ \left( \| X(\omega) \|_{C^{\alpha}}^{- 1}
     \frac{\varepsilon}{\sqrt{2}} \right)^{1 / \alpha},
     \frac{\varepsilon}{\sqrt{2}} \right\} . \]
  Therefore,
  \begin{equation}
    \tmop{Leb} \{ s \in [0, T] : (s, X (\omega)) \in B_{\varepsilon} (K) \}
    \gtrsim_{\| X (\omega) \|_{C^{\alpha}}} \varepsilon^{1 / \alpha} .
    \label{lowerboundLeb}
  \end{equation}
  We are done once we can show that for all $N > 0$,
  \begin{equation}
    \mathbb{P} (\{ \tau^{\infty} < \infty \} \cap \{ \| X \|_{C^{\alpha}} < N
    \}) = 0 . \label{normcondstat}
  \end{equation}
  By (\ref{lowerboundLeb}) we have
  \[ \{ \tau^{\infty} < \infty \} \cap \{ \| X \|_{C^{\alpha}} < N \} \subset
     \bigcap_{0 < \varepsilon < \varepsilon_0 (N)} \{\tmop{Leb} \{t\in[0,T]: (t,X_t)
     \in B_{\varepsilon} (K) \} \geqslant C (N) \varepsilon^{1 / \alpha}\}, \]
  for some constants $C (N), \varepsilon_0 (N)$. On the other hand by
  incompressibility for any $\kappa > 0$ there exists $M (\kappa)$ such that
  \begin{eqnarray*}
    \mathbb{P} (\tmop{Leb} \{ t \in [0, T] : (t, X_t ) \in
    B_{\varepsilon} (K) \} > \varepsilon^{1 / \alpha + \delta / 2}) & = &
    \mathbb{P} \left( \int_0^T 1_{B_{\varepsilon} (K)} (s, X_s) \mathd s >
    \varepsilon^{1 / \alpha + \delta / 2} \right)\\
    & \leqslant & \kappa + M (\kappa) \varepsilon^{\delta / 2},
  \end{eqnarray*}
  meaning that for all $\varepsilon$ small enough such that
  $\varepsilon^{\delta / 2} < C (N) \wedge \varepsilon_0 (N)$,
  \begin{eqnarray*}
    \mathbb{P} (\{ \tau^{\infty} < \infty \} \cap \{ \| X \|_{C^{\alpha}} < N
    \}) & \leqslant & \mathbb{P} (\tmop{Leb} \{ t \in [0, T] : (t, X_t ) \in
    B_{\varepsilon} (K) \}\geqslant C (N) \varepsilon^{1 / \alpha})\\
    & \leqslant & \mathbb{P} (\tmop{Leb} \{ t \in [0, T] : (t, X_t ) \in
    B_{\varepsilon} (K) \} > \varepsilon^{1 / \alpha + \delta / 2})\\
    & \leqslant & \kappa + M (\kappa) \varepsilon^{\delta / 2} .
  \end{eqnarray*}
  Letting first $\varepsilon \rightarrow 0$ and then $\kappa \rightarrow 0$ we
  indeed obtain indeed (\ref{normcondstat}).
\end{proof}

\begin{proof*}{Proof of Theorem \ref{maintheorem}.}
  Recall that
  \[ \rho^{\gamma} \ast u (t \wedge \tau^{\delta}, X_{t \wedge \tau^{\delta}})
     - \rho^{\gamma} \ast u (0, X_0) - \int_0^{t \wedge \tau^{\delta}}
     (\partial_t + \mathLaplace + b \cdummy \nabla) \rho^{\gamma} \ast u (s,
     X_s) \mathd s, t \in [0, T'], \]
  is a continuous martingale in the filtration generated by $X$. Moreover
  \begin{eqnarray}
    \int_0^{t \wedge \tau^{\delta}} (\partial_t + \mathLaplace + b \cdummy
    \nabla) \rho^{\gamma} \ast u (s, X_s) \mathd s & =_{\left( \ref{convint}
    \right)} & \int_0^{t \wedge \tau^{\delta}}  \nabla \cdummy
    (g_{\delta} A \nabla \rho^{\gamma} \ast u)  (s, X_s)  \nonumber\\
    &  & - \rho^{\gamma} \ast  \nabla \cdummy (g_{\delta} A
    \nabla u) (s, X_s) \mathd s \nonumber\\
    & \rightarrow 0,  \label{addfunclimit}
  \end{eqnarray}
  in $ucp$ w.r.t. $t$ since $ \nabla \cdummy (g_{\delta} A \nabla
  \rho^{\gamma} \ast u)  \rightarrow \nabla \cdummy (g_{\delta} A \nabla
   u), \rho^{\gamma} \ast  \nabla \cdummy (g_{\delta} A
    \nabla u) \rightarrow   \nabla \cdummy (g_{\delta} A
    \nabla u)$ in $L^2_T H^{- 1}$ which was shown
  in the proof of Lemma \ref{stoppedob}. The convergences $\rho^{\gamma} \ast u (0,
  X_0) \rightarrow u (0, X_0)$, $\rho^{\gamma} \ast u
  (\cdummy \wedge \tau^{\delta}, X_{\cdummy \wedge \tau^{\delta}}) \rightarrow
  u (\cdummy \wedge \tau^{\delta}, X_{\cdummy \wedge \tau^{\delta}})$ follow
  from Lemma \ref{stoppedob} as well.
  
  Hence, by {\cite[Proposition~IX.1.17]{Jacod2003}}, we obtain that
  \[ u (t \wedge \tau^{\delta}, X_{t \wedge \tau^{\delta}}) - u (0, X_0), t
     \geqslant 0, \]
  is a continuous local martingale which is a proper martingale since it is
  bounded by Lemma \ref{stoppedob}. Now we let $\delta \rightarrow 0$ and get
  that
  \[ u (t \wedge \tau^{\delta}, X_{t \wedge \tau^{\delta}}) \rightarrow u (t,
     X_t), \mathbb{P}- a.s., \]
  and in $L^p$ for any $1 \leqslant p < \infty$ by (\ref{stoppedident}) and
  since (\ref{stopinftyinfty}) holds. In particular
  \[ u (t, X_t) - u (0, X_0), t \in [0, T'], \]
  is a martingale, meaning
  \[ \mathbb{E} [v_0 (X_{T' })] =\mathbb{E} [u (T', X_{T' })] =\mathbb{E} [u
     (0, X_0)] =\mathbb{E} [v (T', X_0)] . \]
  This shows uniqueness of one-dimensional distributions and uniqueness of
  finite-dimensional distributions and the Markov property follow the same
  iterative standard argument as in {\cite{GraefnerPerkowski24}}.
  
  The second statement follows from Lemma \ref{relationdimHoelder} and the
  existence of energy solutions $X \in C^{\alpha}$ with $\alpha < \frac{p (1 /
  2 - 1 / q) - 1}{p}$ for any initial distribution $\mu \ll \tmop{Leb}$
  follows from the application of Kolmogorov's continuity theorem in the proof of Theorem 19 in {\cite{GraefnerPerkowski24}}.
\end{proof*}

\section{Weak convergence rates for the approximations}\label{ratessection}
We are interested in rates for the convergence $\mathbb{E} [f (X_t^n)]
\rightarrow \mathbb{E} [f (X_t)]$, where $f : \mathbb{R}^d \rightarrow \mathbb{R}$ is a suitable function and $X^n$ solves the approximated system
\begin{equation}
\mathd X^n_t = b^n (X^n_t) \mathd t + \sqrt{2} \mathd B_t, \label{approx}
\end{equation}
for $b^n = \rho^n \ast b, n \in \mathbb{N}$ with $\rho^n=n^d\rho(n\cdot)$ and $\rho$ being a smooth mollifier with $\text{supp}\rho\subset B_1(0)$. In particular, we consider
time-independent drift $b$.
\\
\\
We work with inhomogeneous Besov spaces $B^s_{p,q}$, $s\in\mathbb{R}$, $p,q\in [1,\infty]$ with norm \[
\|u \|_{B^{s }_{p, q}}^q=\sum_{j=-1}^\infty 2^{jsq}\|\Delta_j u\|_{L^p}^q,
\]
with the usual interpretation for $q=\infty$ and where $\Delta_j u=\phi_j\ast u$ is the $j$th Littlewood-Paley block for $\phi_j=\mathcal{F}^{-1}\varphi_j$ and the dyadic partition of unity $(\varphi_j)_{j\geqslant -1}$.
\\
The following rather well known Lemma is the source of our rates and says that convergence of mollifications of a distribution to the original object should have a rate if one measures with a weaker norm.

\begin{lemma}
\label{mollificationLemma}
  Let $s \in \mathbb{R}$, $p, q \in [1, \infty]$, let $b \in B^s_{p, q}$
  and consider $b^n = \rho^n \ast b, n \in \mathbb{N}$. Then, for any
  $\alpha>0$ it holds
  \[ \| b - b^n \|_{B^{s - \alpha}_{p, q}} \lesssim n^{- \alpha
     } \| b \|_{B^s_{p, q}} . \]
\end{lemma}

The following Lemma now allows to obtain rates for the convergence $b^n
\cdummy \nabla \rightarrow b \cdummy \nabla$ w.r.t. the regularities that are
given to us by our setting. We refer to {\cite{vanZuijlen22}} for the definition of para- and
resonant products $\varolessthan,\varogreaterthan,\odot$.

\begin{lemma}
  \label{driftinterpolboundlemma}Let $b \in B^{- \gamma}_{p, \infty} (\mathbb{R}^d ,
  \mathbb{R}^d)$ be divergence-free with $p > \frac{2}{1 - \gamma}, \gamma \in
  [0, 1), p \in (2, \infty)$. Let furthermore $u \in L^{\infty}  \cap H^1
  $. Then we have that
  \[ b \cdummy \nabla u := b \varolessthan \nabla u + b
     \varogreaterthan \nabla u + \nabla \cdummy (b \odot u), \]
  satisfies for all $r \in \left( \frac{2 p}{p + 2}, \frac{2 p}{p \gamma + 2}
  \right)$ the bound
  \begin{equation}
    \| b \cdummy \nabla u \|_{B^{- 1}_{r, 2}} \lesssim \| b \|_{B^{-
    \gamma}_{p, \infty} } \| u \|_{L^{\infty} \cap H^1},
    \label{driftinterpolbound}
  \end{equation}
  where
  \[ \| u \|_{L^{\infty} \cap H^1} = \max \{ \| u \|_{L^{\infty}},
     \| u \|_{H^1} \} . \]
\end{lemma}

\begin{proof}
  For $q \in [2, \infty]$ we use the following interpolation bound, which is
  expected to exist somewhere in the literature but was not found since the $L^\infty$
  case is often excluded. We have,
  \begin{eqnarray*}
    \| \mathLaplace_j u \|_{L^q} & = &\left( \int | \mathLaplace_j u |^q
    \right)^{1 / q}\\
    & \leqslant &\| \mathLaplace_j u \|_{L^{\infty}}^{1 - 2 / q} \|
    \mathLaplace_j u \|^{2 / q}_{L^2}\\
    & \lesssim &\| u \|_{L^{\infty}}^{1 - 2 / q} 2^{- 2 j / q} \| u \|^{2 /
    q}_{H^1} .
  \end{eqnarray*}
  Here, we used that $\| \phi_j \|_{L^1} = \| \phi_0 \|$, so that
  \[ \| u \|_{B^{2 / q}_{q, \infty}} \lesssim \| u \|_{L^{\infty} \cap H^1} .
  \]
  With Theorem 27.10 from {\cite{vanZuijlen22}}, for
  $q \geqslant 2$ small enough such that
  \[ \frac{2}{q} > \gamma, \]
  we have
  \[ \| \nabla \cdummy (b \odot u) \|_{B^{- 1 - \gamma + \frac{2}{q}}_{r,
     \infty}} \lesssim \| b \|_{B^{- \gamma}_{p, \infty}} \| u \|_{B^{2 /
     q}_{q, \infty}}, \]
  where
  \[ \frac{1}{r} = \frac{1}{p} + \frac{1}{q} \leqslant 1, \]
  and the last inequality follows from $q \geqslant 2$ and $p > \frac{2}{1 -
  \gamma} \geqslant 2$. Similarly, with Theorems 27.5 from {\cite{vanZuijlen22}} we bound
  \[ \| b \varolessthan \nabla u \|_{B^{- 1 - \gamma + \frac{2}{q}}_{r,
     \infty}} \lesssim \| b \|_{B^{- \gamma}_{p, \infty}} \| \nabla u \|_{B^{-
     1 + 2 / q}_{q, \infty}} \lesssim \| b \|_{B^{- \gamma}_{p, \infty}} \| u
     \|_{B^{2 / q}_{q, \infty}}, \]
  and assuming $q > 2$
  \[ \| b \varogreaterthan \nabla u \|_{B^{- 1 - \gamma + \frac{2}{q}}_{r,
     \infty}} \lesssim \| b \|_{B^{- \gamma}_{p, \infty}} \| \nabla u \|_{B^{-
     1 + 2 / q}_{q, \infty}} \lesssim \| b \|_{B^{- \gamma}_{p, \infty}} \| u
     \|_{B^{2 / q}_{q, \infty}} . \]
  Overall, we deduce that
  \[ \| b \cdummy \nabla u \|_{B^{- 1}_{r, 2}} \lesssim \| b \cdummy \nabla u
     \|_{B^{- 1 - \gamma + \frac{2}{q}}_{r, \infty}} \lesssim \| b \|_{B^{-
     \gamma}_{p, \infty}} \| u \|_{B^{2 / q}_{q, \infty}}, \]
  where $\frac{1}{r} = \frac{1}{p} + \frac{1}{q}$ and $q \in \left( 2,
  \frac{2}{\gamma} \right)$, meaning $r \in \left( \frac{2 p}{p + 2}, \frac{2
  p}{p \gamma + 2} \right)$.
\end{proof}

We will also need the following Krylov-type \textit{It$\hat{o}$ trick} bound from Lemma 16 in {\cite{GraefnerPerkowski24}}. Our version is slightly more general since it allows initial densities $\eta$ with finite integrability index; but the proof follows the same steps as the one in the $L^\infty$-case in {\cite{GraefnerPerkowski24}}.
\begin{lemma}[It{\^o} trick]
  \label{ItotricklemmaEuclid}Let $b \in C (\mathbb{R}_+ ; C^{\infty}_b
  (\mathbb{R}^d))$ and let $X : \mathbb{R}_+ \rightarrow \mathbb{R}^d$ solve
  \[ \mathd X_t = b (t, X_t) \mathd t + \sqrt{2} \mathd B_t, \]
  where $B$ is a Brownian motion, $\nabla \cdummy b \equiv 0$ and $b, \nabla b
  \in L^{\infty}_{\tmop{loc}} L^s$ for some $s \in [1, \infty)$. Assume that
  $X (0) \sim \mu$ with $\eta \assign \frac{d \mu}{d
  \tmop{Leb}_{\mathbb{R}^d}} \in L^{\kappa} (\mathbb{R}^d)$ for some $\kappa
  \in (1, \infty]$. Let $\kappa' \in [1, \infty)$ such that $\frac{1}{\kappa}
  + \frac{1}{\kappa'} = 1$. Then, for all $T > 0$ and $f \in C^{\infty}_c ([0,
  T] \times \mathbb{R}^d, \mathbb{R})$
  \begin{equation}
    \mathbb{E} \left[ \sup_{0 \leqslant t \leqslant T} \left| \int_0^t
    \mathLaplace f (s, X_s) \mathd s \right|^p \right] \lesssim \| \eta
    \|_{L^{\kappa} (\mathbb{R}^d)} T^{p \left( \frac{1}{2} - \frac{1}{q}
    \right)} \| \nabla f \|^p_{L_T^q L^{p \kappa'} (\mathbb{R}^d)},
    \label{ItotrickEuclid}
  \end{equation}
  for all $p \in \left[ \frac{2}{\kappa'}, \infty \right)$ and $q \in [2,
  \infty]$. The implicit constant on the right hand side is independent of
  $b$.
\end{lemma}
Inverting the Laplacian, we obtain, under the same assumptions as in the previous Lemma,
\begin{eqnarray*}
\mathbb{E} \left[ \sup_{0 \leqslant t \leqslant T} \left| \int_0^t
    f (s, X_s) \mathd s \right|^p \right]&\lesssim &\mathbb{E} \left[ \sup_{0 \leqslant t \leqslant T} \left| \int_0^t
    \Delta (1-\Delta)^{-1}f (s, X_s) \mathd s \right|^p \right]\\
    &&+\mathbb{E} \left[ \sup_{0 \leqslant t \leqslant T} \left| \int_0^t
    (1-\Delta)^{-1}f (s, X_s) \mathd s \right|^p \right] .
\end{eqnarray*}
Since $\nabla\cdot b=0$ we have that $\|\eta_t \|_{L^\kappa}\leqslant \|\eta \|_{L^\kappa}$ where $\eta_t$ is the density of $\text{law}(X_t)$ w.r.t. $\text{Leb}$ at time $t$. Thus, with Jensen's inequality and $p\geqslant 1$ and using the continuous embedding $B^0_{r,2}\subset L^r$ for finite $r\geqslant 2$,
\[
\mathbb{E} \left[ \sup_{0 \leqslant t \leqslant T} \left| \int_0^t
    (1-\Delta)^{-1}f (s, X_s) \mathd s \right|^p \right]\lesssim \|\eta \|_{L^\kappa}T^{p-1}\|f\|_{L_T^1 B^{-2}_{\kappa 'p,2}}^p .
\]
In total
\begin{equation}
\label{usefulItotrick}
\mathbb{E} \left[ \sup_{0 \leqslant t \leqslant T} \left| \int_0^t
    f (s, X_s) \mathd s \right|^p \right]\lesssim_T \|\eta \|_{L^\kappa}\|f\|_{L_T^q B^{-1}_{\kappa 'p,2}}^p .
\end{equation}

Now we are in the position to derive weak convergence rates.

\begin{theorem}
\label{ratewithoutlocalsingularities}
  Let $b \in B^{- \gamma}_{p, \infty}, \nabla \cdummy b = 0$ where $\gamma
  \in [0, 1), p \in (2, \infty)$ and $p > \frac{2}{1 - \gamma}$. Let $\mu \ll
  \tmop{Leb}$, let $X$ be the unique energy solution to (\ref{mainSDE}) with
  initial distribution $\mu$ and denote by $X^n$ the corresponding solutions
  to the approximated system (\ref{approx}). Let $\beta \in \left[ 0, 1 -
  \frac{2}{p} - \gamma \right)$ and assume that $\eta \assign \frac{\mathd
  \mu}{\mathd \tmop{Leb}} \in L^q$ with $q \in \left( \frac{2 p}{(2 -
  \beta - \gamma) p - 2}, \infty \right]$. Then, for any $f \in L^2 \cap
  L^{\infty}$ it holds that
  \[ \sup_{t \in [0, T]}  | \mathbb{E} [f (X_t) - f (X_t^n)] | \lesssim_T \max
     \{ 1, \| \eta \|_{L^q} \} {{n^{- \beta}} }  \| b \|_{B^{- \gamma}_{p,
     \infty}} \| f \|_{L^2 \cap L^{\infty}} . \]
\end{theorem}
\begin{remark}
\label{rateremark}
Note that the rate $\beta \in \left[ 0, 1 -\frac{2}{p} - \gamma \right)$ for $p > \frac{2}{1 - \gamma}$ and $b \in B^{- \gamma}_{p, \infty}$ in Theorem \ref{ratewithoutlocalsingularities} gets arbitrarily close to $\beta=1$ for large $p$ and small $\gamma>0$ while the drift is still supercritical in large enough dimensions since the threshold for criticality is $p=\frac{d}{1-\gamma}$.
\end{remark}

\begin{proof}
  \
  
  Let $P^n_t g (x_0) =\mathbb{E}_{x_0} [g (X_t^n)] $, where $X^n$ is the law
  under which $X^n (0) = x_0$, be the strongly continuous semigroup on $C^k_0$ for any $k \in \mathbb{N}$. In particular, we have the ``Duhamel formula'',
  \begin{equation}
    P_t^m - P_t^n = \int_0^t P^n_{t - s} (b^m - b^n) \cdummy \nabla P_s^m
    \mathd s = \int_0^t P^n_s (b^m - b^n) \cdummy \nabla P_{t - s}^m \mathd s
    . \label{Duhamelforrate}
  \end{equation}
  Moreover, we have the $m$-uniform bound \[
  \|P^mf||_{L_T^2H^1\cap L_T^\infty L^\infty}\leqslant \| f \|_{L^2 \cap L^{\infty}},
  \]
  which follows from the Markov property of $X^m$ and a simple PDE energy estimate using that $b$ is divergence-free (see Lemma 21 in {\cite{GraefnerPerkowski24}}).
  With (\ref{usefulItotrick}) we can bound for $\kappa = q \wedge 2
  \in (1, 2]$, $\frac{1}{\kappa} + \frac{1}{\kappa'} = 1$, and wolog $f \in
  \mathcal{S}$,
  \begin{eqnarray*}
    | \mathbb{E} [f (X_t) - f (X_t^n)] | & = & \lim_m | \mathbb{E} [f (X_t^m)
    - f (X_t^n)] |\\
    & = & \lim_m | \langle \eta, (P_t^m - P_t^n) f \rangle |\\
    & = & \lim_m \left| \langle \eta, \int_0^t P^n_s (b^m - b^n) \cdummy
    \nabla P_{t - s}^m f \rangle \mathd s \right|\\
    & = & \lim_m \left| \mathbb{E}_{X^n (0) \sim \eta \mathd \tmop{Leb}}
    \left[ \int_0^t (b^m - b^n) \cdummy \nabla P_{t - s}^m f (X_s^n) \mathd s
    \right] \right|\\
    & \lesssim & \limsup_m \| \eta \|_{L^{\kappa}} \| (b^m - b^n) \cdummy
    \nabla P^m_{t - \cdummy} f \|_{L^2_T B^{- 1}_{\kappa', 2}}\\
    & \lesssim & \limsup_m \| \eta \|_{L^{\kappa}} {{n^{- \beta +
    \varepsilon}} }  \| b \|_{B^{- \gamma}_{p, \infty}} \| P^m_{\cdummy} f
    \|_{L^2_T (H^1 \cap L^{\infty})}\\
    & \lesssim & \| \eta \|_{L^{\kappa}} {{n^{- \beta + \varepsilon}} }  \| b
    \|_{B^{- \gamma}_{p, \infty}} \| f \|_{L^2 \cap L^{\infty}},
  \end{eqnarray*}
  provided $\kappa' \in \left( \frac{2 p}{p + 2}, \frac{2 p}{p (\beta +
  \gamma) + 2} \right)$. Since $\frac{2 p}{p + 2} < 2 \leqslant \kappa'$, this
  holds already if $\kappa > \frac{2 p}{(2 - \beta - \gamma) p - 2} > 2$,
  which is true. Bounding $\| \eta \|_{L^{\kappa}} \lesssim \| \eta \|_{L^1
  \cap L^q} = \max \{ 1, \| \eta \|_{L^q} \}$ by interpolation we obtain
  \[ \sup_{t \in [0, T]}  | \mathbb{E} [f (X_t) - f (X_t^n)] | \lesssim \max
     \{ 1, \| \eta \|_{L^q} \} {{n^{- \beta + \varepsilon}} }  \| b \|_{B^{-
     \gamma}_{p, \infty}} \| f \|_{L^2 \cap L^{\infty}} . \]
  Since the claim is true for any $q, \beta$ satisfying $\beta \in \left[ 0, 1
  - \frac{2}{p} - \gamma \right)$ and $q \in \left( \frac{2 p}{(2 - \beta -
  \gamma) p - 2}, \infty \right]$, we can take $\tilde{\beta} = \beta +
  \varepsilon$ for $\varepsilon$ small enough. The case for general $f$
  follows from simple approximation arguments.
\end{proof}

Arguing with stopping times we can treat the setting of Theorem
\ref{maintheorem}, i.e. drift with local singularities. For the purpose of a
not too complicated presentation, we assume that $\eta = \frac{\mathd
\tmop{law} (X_0)}{\mathd \tmop{Leb}} \in L^{\infty}$ and that for any
$\varepsilon > 0$, we have that $A g_{\varepsilon}$ lies in a certain
supercritical Besov space with fixed parameters instead of the larger class
$\mathcal{A}$. The following result gives a semi-explicit rate which depends
on the explosive behaviour of $A$ around the singularity.

\begin{corollary}
  \label{ratesingularitycorollary}Let $b = \mathbf{b} (A)$ where $A \in L^p
  (\mathbb{R}^d)$ for some $p > 2$. Assume that there exists a compact set $K
  \subset \mathbb{R}^d$ such that for some $\delta > 0$
  \[ \tmop{Leb}_{\mathbb{R}^d} (B_{\varepsilon} (K)) \lesssim \varepsilon^{2 p
     / (p - 2) + \delta}, \varepsilon > 0 . \]
  Assume further that there exist $\gamma \in [0, 1), p \in (2, \infty)$ with
  $p > \frac{2}{1 - \gamma}$ as well as a family
  $(g_{\varepsilon})_{\varepsilon > 0} \subset C^{\infty} (\mathbb{R}^d)$ with
  \[ g_{\varepsilon} (x) = \left\{\begin{array}{ll}
       1, & d (x, K) > \varepsilon / 2,\\
       0, & d (x, K) < \varepsilon / 4,
     \end{array}\right., \]
  such that
  \[ A g_{\varepsilon} \in B^{- \gamma + 1}_{p, \infty}, \varepsilon > 0 . \]
  Let $\mu \ll \tmop{Leb}_{\mathbb{R}^d}$ with $\eta \assign \frac{\mathd
  \mu}{\mathd \tmop{Leb}} \in L^{\infty}$, let $X$ be the unique energy
  solution to (\ref{mainSDE}) with initial distribution $\mu$ and denote by
  $X^n$ the corresponding solutions to the approximated system (\ref{approx}).
  Let $\beta \in \left[ 0, 1 - \frac{2}{p} - \gamma \right)$. Then, for $f \in
  L^2 \cap L^{\infty}$ it holds for all $\kappa > 0$ that
  \[ \sup_{t \in [0, T]} | \mathbb{E} [f (X_t)] - [f (X_t^n)] | \lesssim_T
     \max \{ 1, \| \eta \|_{L^{\infty}} \} {{n^{- \beta}} }  \|
     b^{\varepsilon} \|_{B^{- \gamma}_{p, \infty}} \| f \|_{L^2 \cap
     L^{\infty}} + \varepsilon^{\delta \frac{p - 2}{p} - \kappa} \| f
     \|_{L^{\infty}}, \]
  uniformly in $\varepsilon \in (0, 1)$, where $b^{\varepsilon} = \mathbf{b}
  (g_{\varepsilon} A)$.
\end{corollary}

\begin{proof}
  From the uniqueness of energy solutions in Theorem \ref{maintheorem} it
  follows that
  \[ \tmop{law} ((X_{t \wedge \tau^{\varepsilon} (X)})_{t \geqslant 0}) =
     \tmop{law} ((X^{\varepsilon}_{t \wedge \tau^{\varepsilon}
     (X^{\varepsilon})})_{t \geqslant 0}), \]
  where $X^{\varepsilon}$ is the energy solution corresponding to the drift
  $b^{\varepsilon} = \mathbf{b} (g_{\varepsilon} A)$ and $\tau^{\varepsilon}
  (X)$ is defined as in the setting of Theorem \ref{maintheorem}. Indeed, any
  energy solution is equal in law to the energy solution constructed from
  approximations in Theorem 19 in {\cite{GraefnerPerkowski24}}. We would like to conclude our claim from
  the fact that
  \begin{equation}
    \tmop{law} ((X^n_{t \wedge \tau^{\varepsilon} (X^n)})_{t \geqslant 0}) =
    \tmop{law} ((X^{\varepsilon, n}_{t \wedge \tau^{\varepsilon}
    (X^{\varepsilon, n})})_{t \geqslant 0}), \label{stoppedapproxequality}
  \end{equation}
  where $X^n$ is the solution corresponding to the drift $b^n = \mathbf{b}
  (\rho^n \ast A)$ while $X^{\varepsilon, n}$ corresponds to $b^{\varepsilon,
  n} = \mathbf{b} (g_{\varepsilon} \rho^n \ast A)$. With the arguments from
  the proof of Theorem 19 in {\cite{GraefnerPerkowski24}} it is straightforward to verify that, even though
  the approximation has changed slightly, $X^{\varepsilon, n}$ still converges
  in law on $C ([0, T], \mathbb{R}^d)$ to $X^{\varepsilon}$. The problem now
  is that the map
  \[ F : C ([0, T], \mathbb{R}^d) \rightarrow C ([0, T], \mathbb{R}^d),
     (Y_t)_{t \in [0, T]} \mapsto (Y_{t \wedge \tau^{\varepsilon} (Y)})_{t \in
     [0, T]}, \]
  is not continuous and hence we can not directly conclude
  convergence in law on both sides of (\ref{stoppedapproxequality}). However,
  we can approximate $F$ by continuous functions: Let $(h_N)_{N \in \mathbb{N}} \subset C^1
  (\mathbb{R}_+, \mathbb{R})$ with be such that
  \[ h_N (r) = \left\{\begin{array}{ll}
       1, & r \geqslant \frac{1}{N},\\
       r, & r < \frac{1}{2 N},
     \end{array}\right. . \]
  For $Y \in C ([0, T], \mathbb{R}^d)$ we define
  \[ p^N (t) = \int_0^t 1 \wedge (h_N (\min_{0 \leqslant r \leqslant s} d
     (Y_r, \overline{B_{\varepsilon} (K)}))), \quad t \in [0, T] . \]
  It is straightforward to verify that
  \[ F^N : C ([0, T], \mathbb{R}^d) \rightarrow C ([0, T], \mathbb{R}^d),
     (Y_t)_{t \in [0, T]} \mapsto (Y_{ p^N (t)})_{t \in [0, T]}, \]
  is continuous so that
  \[ \tmop{law} (F^N (X)) = \tmop{law} (F^N (X^{\varepsilon})) . \]
  Moreover, we have that $p^N (t) = t$ for $t \leqslant \sigma^N (Y) \assign
  \inf \left\{ s \in [0, T] : d (Y_s, \overline{B_{\varepsilon} (K)})
  \leqslant \frac{1}{N} \right\}$ and thus $F^N (Y)_t = F (Y)_t$ for $t
  \leqslant \sigma^N (Y)$. Since $\sigma^N (Y) \uparrow \tau^{\varepsilon}$ we
  get that $F^N (Y) \rightarrow F (Y)$ in $C ([0, T], \mathbb{R}^d)$ for any
  $Y$. Thus, we obtain (\ref{stoppedapproxequality}) by taking $N \rightarrow
  \infty$.
  
  \
  
  Let $f \in \mathcal{S}$ for now. Then,
  \[ | \mathbb{E} [f (X_t)] -\mathbb{E} [f (X^{\varepsilon}_{t \wedge
     \tau^{\varepsilon} (X^{\varepsilon})})] | = | \mathbb{E} [f (X_t)]
     -\mathbb{E} [f (X_{t \wedge \tau^{\varepsilon} (X)})] | \leqslant \| f
     \|_{L^{\infty}} \mathbb{P} (\tau^{\varepsilon} (X) < t), \]
  with the same holding for $X^n$ and $X^{\varepsilon, n}$ and therefore,
  \begin{eqnarray*}
    | \mathbb{E} [f (X_t)] - [f (X_t^n)] | & \leqslant & \| f \|_{L^{\infty}}
    \{ \mathbb{P} (\tau^{\varepsilon} (X) < t) +\mathbb{P} (\tau^{\varepsilon}
    (X^n) < t) \}\\
    &  & + | \mathbb{E} [f (X^{\varepsilon}_{t \wedge \tau^{\varepsilon}
    (X^{\varepsilon})})] -\mathbb{E} [f (X^{\varepsilon, n}_{t \wedge
    \tau^{\varepsilon} (X^{\varepsilon, n})})] |\\
    & \leqslant & \| f \|_{L^{\infty}}\{ \mathbb{P} (\tau^{\varepsilon} (X) < t) +\mathbb{P}
    (\tau^{\varepsilon} (X^n) < t)\\
    &  &  ...+\mathbb{P} (\tau^{\varepsilon}
    (X^{\varepsilon}) < t) +\mathbb{P} (\tau^{\varepsilon} (X^{\varepsilon,
    n}) < t) \}\\
    &  & + | \mathbb{E} [f (X^{\varepsilon}_t)] -\mathbb{E} [f
    (X^{\varepsilon, n}_t)] |\\
    & \lesssim & \max \{ 1, \| \eta \|_{L^\infty} \} {{n^{- \beta}} }  \|
    b^{\varepsilon} \|_{B^{- \gamma}_{p, \infty}} \| f \|_{L^2 \cap
    L^{\infty}}\\
    &  & + \| f \|_{L^{\infty}} \{ \mathbb{P} (\tau^{\varepsilon} (X) < t) +\mathbb{P}
    (\tau^{\varepsilon} (X^n) < t)\\
    &&...+\mathbb{P} (\tau^{\varepsilon}
    (X^{\varepsilon}) < t) +\mathbb{P} (\tau^{\varepsilon} (X^{\varepsilon,
    n}) < t) \} .
  \end{eqnarray*}
  Now, we derive a bound for $\mathbb{P} (\tau^{\varepsilon} (X) < t),
  \mathbb{P} (\tau^{\varepsilon} (X^n) < t), \mathbb{P} (\tau^{\varepsilon}
  (X^{\varepsilon}) < t), \mathbb{P} (\tau^{\varepsilon} (X^{\varepsilon, n})
  < t)$ which will be the same for all four terms, whence we focus on the
  first one. For this purpose, we follow the arguments from the proof of Lemma
  \ref{relationdimHoelder}: Let $\omega \in \Omega$, such that
  $\tau^{\varepsilon} (X) (\omega) < \infty$, i.e. for some $t \in [0, T]$, we
  have $X_t  (\omega) \in \overline{B_{\varepsilon} (K)}$. Then, still $X_s 
  (\omega) \in B_{2 \varepsilon} (K)$ if, with $0 < \alpha < \frac{1}{2} -
  \frac{1}{p}$,
  \[ | t - s | < (\| X (\omega) \|_{C^{\alpha}}^{- 1} \varepsilon)^{1 /
     \alpha}, \]
  so that for any $L > 0$ and for $\varepsilon$ small enough
  \begin{eqnarray}
    \mathbb{P} (\tau^{\varepsilon} (X) < \infty) & \leqslant & \mathbb{P}
    (\tmop{Leb} \{ t \in [0, T] : X_t \in B_{2 \varepsilon} (K) \} \geqslant
    (\| X \|_{C^{\alpha}}^{- 1} \varepsilon)^{1 / \alpha}) \nonumber\\
    & \leqslant & \mathbb{P} \left( \tmop{Leb} \{ t \in [0, T] : X_t \in B_{2
    \varepsilon} (K) \} \geqslant \left( \frac{\varepsilon}{L} \right)^{1 /
    \alpha} \right) \nonumber\\
    &  & +\mathbb{P} (\| X \|_{C^{\alpha}} > L) . 
    \label{estimatestoppingtime}
  \end{eqnarray}
  To estimate the first term, we cannot use the general incompressibility
  estimate for $X$ since the constants therein do not provide an explicit
  rate. Instead we argue directly and obtain
  \begin{eqnarray*}
    &&\mathbb{P} \left( \tmop{Leb} \{ t \in [0, T] : X_t \in B_{2 \varepsilon}
    (K) \} \geqslant \left( \frac{\varepsilon}{L} \right)^{1 / \alpha}
    \right)\\
    &=&\mathbb{P} \left( \int_0^T 1_{B_{2 \varepsilon} (K)} (X_s) \mathd s
    \geqslant \left( \frac{\varepsilon}{L} \right)^{1 / \alpha} \right)\\
    &\leqslant & \left( \frac{\varepsilon}{L} \right)^{- 1 / \alpha} \mathbb{E}
    \left[ \int_0^T 1_{B_{2 \varepsilon} (K)} (X_s) \mathd s \right]\\
    &=& \left( \frac{\varepsilon}{L} \right)^{- 1 / \alpha} \int_0^T \langle
    \eta, P_s 1_{B_{2 \varepsilon} (K)} \rangle \mathd s\\
    &\leqslant & \left( \frac{\varepsilon}{L} \right)^{- 1 / \alpha} \| \eta
    \|_{L^{\infty}} \int_0^T \| P_s 1_{B_{2 \varepsilon} (K)} \|_{L^1} \mathd
    s\\
    &\leqslant &\left( \frac{\varepsilon}{L} \right)^{- 1 / \alpha} \| \eta
    \|_{L^{\infty}} \int_0^T \| 1_{B_{2 \varepsilon} (K)} \|_{L^1} \mathd s\\
    &\lesssim_T &\left( \frac{\varepsilon}{L} \right)^{- 1 / \alpha} \tmop{Leb}
    (B_{2 \varepsilon} (K))\\
    &\lesssim  &\varepsilon^{\delta + 2 p / (p - 2) - 1 / \alpha} L^{1 / \alpha}
    .
  \end{eqnarray*}
  For the second term in (\ref{estimatestoppingtime}) we refer to the proof
  of Theorem 19 in {\cite{GraefnerPerkowski24}}, where it was shown that
  \[ \mathbb{E} [| X_t - X_s |^p] \lesssim_{\| \eta \|_{L^{\infty}}, b} | t -
     s |^{p / 2}, \]
  so that we obtain, as already noted in the proof of Theorem
  \ref{maintheorem}, that almost surely $X \in C^{(p - 2) / (2 p) -}$.
  However, to obtain a rate, we need to use a probabilistic estimate for the
  H{\"o}lder norm of $X$ as in Theorem 4.3.2 in {\cite{Stroock10}} which
  yields
  \[ \sup_{\varepsilon, n} \mathbb{E} [\| X^{\varepsilon, n}
     \|_{C^{\alpha}}^p] +\mathbb{E} [\| X^n \|_{C^{\alpha}}^p] +\mathbb{E} [\|
     X^{\varepsilon} \|_{C^{\alpha}}^p] +\mathbb{E} [\| X \|_{C^{\alpha}}^p] <
     \infty, \alpha < \frac{1}{2} - \frac{1}{p} . \]
  (Here we included the bounds for $X^{\varepsilon, n}$ etc. since we need
  them to be uniform, which is true). Thus by Markov's inequality,
  \[ \mathbb{P} (\| X \|_{C^{\alpha}} > L) \lesssim L^{- p} . \]
  In total we obtain
  \begin{eqnarray*}
    | \mathbb{E} [f (X_t)] - [f (X_t^n)] | & \lesssim & \max \{ 1, \| \eta
    \|_{L^\infty} \} {{n^{- \beta}} }  \| b^{\varepsilon} \|_{B^{- \gamma}_{p,
    \infty}} \| f \|_{L^2 \cap L^{\infty}}\\
    &  & + \| f \|_{L^{\infty}} (\varepsilon^{\delta + 2 p / (p - 2) - 1 /
    \alpha} L^{1 / \alpha} + L^{- p}) .
  \end{eqnarray*}
  Let $\alpha$ be large enough such that $\tilde{\delta} = \delta + 2 p / (p -
  2) - 1 / \alpha > 0$. Computing derivatives we see that $L^{1 / \alpha}
  \varepsilon^{\tilde{\delta}} + L^{- p}$ can be minimized as a function of
  $L$ at $L_{\min} (\varepsilon) = \left( \frac{p}{\frac{1}{\alpha}
  \varepsilon^{\tilde{\delta}}} \right)^{\frac{1}{1 / \alpha + p}}$ so that
  \begin{eqnarray*}
    | \mathbb{E} [f (X_t)] - [f (X_t^n)] | & \lesssim & \max \{ 1, \| \eta
    \|_{L^\infty} \} {{n^{- \beta}} }  \| b^{\varepsilon} \|_{B^{- \gamma}_{p,
    \infty}} \| f \|_{L^2 \cap L^{\infty}} + \| f \|_{L^{\infty}}
    \varepsilon^{\frac{\tilde{\delta} \alpha p}{1 + \alpha p}} .
  \end{eqnarray*}
  The claim follows since (theoretically) for $\alpha = \frac{1}{2} -
  \frac{1}{p}$ we have $\frac{\tilde{\delta} \alpha p}{1 + \alpha p} = \delta
  \frac{p - 2}{p}$.
\end{proof}

\section*{Acknowledgments}

The author would like to thank Nicolas Perkowski for fruitful and continued
discussions and is grateful for funding by DFG through EXC 2046, Berlin
Mathematical School, through IRTG 2544 ``Stochastic Analysis in
Interaction'' and through the UKRI Future Leaders Fellowship, 2022 - "Large-scale universal behaviour of Random Interfaces and Stochastic Operators" (G. Cannizzaro).

\section{Appendix}
\begin{proof*}{\textit{Proof of Lemma \ref{mollificationLemma}.}}
  We may wolog assume that $\alpha\leqslant 1$ since for $\alpha>1$, one can just apply the case of $\alpha\leqslant 1$ several times. We have with Young's convolution inequality,
  \[ \| b - b^n \|_{B^{s - \alpha}_{p, q}} = \| (\delta_0 - \rho^n) \ast b
     \|_{B^{s - \alpha}_{p, q}} \lesssim \| b \|_{B^s_{p, q}} \| \delta_0 -
     \rho^n \|_{B^{- \alpha}_{1, \infty}} . \]
  Then one checks that, extending the definition of $\rho^n$ from $n \in
  \mathbb{N}$ to $\mathbb{R}^{> 0}$,
  \begin{eqnarray}
    \| \delta_0 - \rho^n \|_{B^{- \alpha}_{1, \infty}} & = & \sup_{j \geqslant
    -1} 2^{- \alpha j} \| \phi_j - \rho^n \ast \phi_j \|_{L^1} \nonumber\\
    & \lesssim & \sup_{j \geqslant 0} 2^{- \alpha j} \| \phi_0 - \rho^{2^{- j} n}
    \ast \phi_0 \|_{L^1}+  \| \phi_{-1} - \rho^n \ast \phi_{-1} \|_{L^1},  \label{prepidentityrate}
  \end{eqnarray}
  For $2^{j}n^{-1}>1$ we simply bound
  \[
  \| \phi_0 - \rho^{2^{- j} n} \ast \phi_0 \|_{L^1}\lesssim \|\phi_0 \|_{L^1}\lesssim 2^{\alpha j}n^{-\alpha} .
  \]
  For $2^{j}n^{-1}\leqslant 1$, using that $\phi_0$ is a Schwartz function, we have
  \begin{eqnarray*}
  \| \phi_0 - \rho^{2^{- j} n} \ast \phi_0 \|_{L^1}&=&\int \left| \int 2^{-dj} n^d\rho (2^{-j} n
    y) (\phi_0 (x) -\phi_0 (x-y) )  \mathd y \right| \mathd x\\
    &\leqslant & 2^{j\alpha} n^{-\alpha} \int  \sup_{|x-y|\leqslant 1}\frac{|\phi_0 (x)-\phi_0 (y)|}{|x-y|^{\alpha}}\mathd x\\
    &\lesssim & 2^{j\alpha} n^{-\alpha} .
  \end{eqnarray*}
  Arguing in the same way for the $\phi_{-1}$-term, this concludes the proof.
\end{proof*}

\bibliographystyle{alpha}
\bibliography{stopping.bib}

\end{document}